\documentclass[noamsfonts]{amsart}
\usepackage{setspace}
\usepackage[charter,expert]{mathdesign}
\usepackage[osf]{XCharter}
\selectfont
\usepackage[activate={true,nocompatibility},final]{microtype}

\usepackage[all, arc]{xy}
\SelectTips{eu}{}
\entrymodifiers={+!!<0pt,\fontdimen22\textfont2>}

\usepackage[pdftex, colorlinks=true, linkcolor=blue, citecolor=magenta, linktocpage]{hyperref}

\theoremstyle{plain}
\newtheorem{thm}[subsubsection]{Theorem}

\newtheorem{prop}[subsubsection]{Proposition}

\theoremstyle{definition}

\theoremstyle{definition}

\numberwithin{equation}{subsubsection}

\def\11{\mathbf{1}}

\def\AA{\mathbf{A}} 
 
\def\CC{\mathbf{C}}

\def\QQ{\mathbf{Q}}

\def\ZZ{\mathbf{Z}}

\def\coker{\mathrm{coker}}

\def\DM{\mathcal{D}}

\def\Gr{\mathrm{Gr}}

\def\id{\mathrm{id}}

\def\rat{\mathrm{rat}}

\def\Spec{\mathrm{Spec}}

\newcommand{\mapright}[1]{\xrightarrow{#1}}

\makeatletter
\newcommand{\holim@}[2]{%
      \vtop{\m@th\ialign{##\cr
                  \hfil$#1\operator@font holim\,$\hfil\cr
                      \noalign{\nointerlineskip\kern1.5\ex@}#2\cr
                          \noalign{\nointerlineskip\kern-\ex@}\cr}}%
                      }
                      \newcommand{\holim}{%
                            \mathop{\mathpalette\holim@{\leftarrowfill@\textstyle}}\nmlimits@
                        }
                        \makeatother

                        \title{On the Clemens-Schmid exact sequence}
\author{R. Virk}
\begin{document}
\maketitle
\renewcommand{\thesubsection}{\textbf{\arabic{subsection}}}
\renewcommand{\thesubsubsection}{\textbf{\arabic{subsection}.\arabic{subsubsection}}}

The Clemens-Schmid exact sequence \cite{C} relates the complex geometry of a semistable degeneration of varieties to that of a general fibre and the monodromy of the degeneration. It is a well known tool for studying semistable degenerations.

Less widely known is that straightforward yoga with weights yields statements in greater generality (eg., there is no need to restrict to the semistable situation). The purpose of this note is to explain this via Theorem \ref{csthm}, Theorem \ref{invthm} and Theorem \ref{uncsthm}.
We use M. Saito's mixed Hodge modules to talk about weights. The essence of the argument is extracted from \cite[Th\'eor\`eme 3.6.1]{D} (also see \cite[Remark 5.2.2]{S88}). In particular, no claims to originality are being made.

\subsection{Preliminaries on mixed Hodge modules}
\subsubsection{}
For a variety\footnote{`Variety' = `separated scheme of finite type over $\mathrm{Spec}(\CC)$'.} $X$, 
we write $\DM(X)$ for the bounded derived category of mixed Hodge modules on $X$. We will mainly use formal properties of the standard functors between these categories, and their interactions with weights, as explained in \cite{S89}.

\subsubsection{}
Part of the data defining a mixed Hodge module $M$ consists of a perverse sheaf $M_{\rat}$ with $\QQ$-coefficients and a finite increasing filtration, the \emph{weight filtration}, on $M_{\rat}$. Morphisms of mixed Hodge modules are \emph{strictly} compatible with the weight filtration. 
A mixed Hodge module is \emph{pure} if its weight filtration is concentrated in a single weight.
An object $M\in\DM(X)$ is called pure of weight $k$ if its $n$-th cohomology module (corresponding to the $n$-th perverse cohomology of $M_{\rat}$) is pure of weight $n+k$. A mixed Hodge module over a point is a polarizable $\QQ$-mixed Hodge structure with the usual notion of weights/weight filtration from mixed Hodge theory.

\subsubsection{}Given a morphism of varieties $f\colon X\to Y$, there are functors $f_*,f_!, f^*, f^!, \boxtimes$ that lift their counterparts on the underlying derived categories of constructible sheaves. (Note: functors on derived categories will always be derived. I.e., we write $f_*$ instead of $Rf_*$, etc.).
The functors $f_*$ and $f^!$ do not lower weights, $f_!$ and $f^*$ do not raise weights, and $\boxtimes$ adds weights. In particular, if $a\colon X\to \Spec(\CC)$ is proper and $M\in \DM(X)$ is pure of weight $0$, then the cohomology module (i.e., Hodge structure):
\[ H^k(X; M) = H^k(a_*M) \]
 is pure of weight $k$.
 
\subsection{Standard exact sequences}
\subsubsection{}Let $i\colon X_0\to X$ be a closed immersion and $j\colon X-X_0 \to X$ the inclusion of its open complement. Then, for $M\in \DM(X)$, we have a canonical distinguished triangle:
\[ i_*i^!M \to M \to j_*j^*M \mapright{[1]} \]
Applying $i^*$ yields the distinguished triangle:
\[ i^!M \to i^*M \to i^*j_*j^*M \mapright{[1]} \]
Now taking $\ast$-pushforward along $X_0\to \Spec(\CC)$ yields the long exact sequence of mixed Hodge structures:
\begin{equation}\label{relles}
\cdots \to H^k(X_0; i^!M) \to H^k(X_0; i^*M) \to H^k(X_0; i^*j_*j^*M) \to \cdots
\end{equation}

\subsubsection{}Given a morphism $f\colon X \to \AA^1$, we set $X_0 = f^{-1}(0)$ and $X^* = X - X_0$. We write $i\colon X_0 \to X$ and $j\colon X^*\to X$ for the inclusions.
We also have the nearby cycles functor $\psi_f\colon \DM(X) \to\DM(X_0)$
that lifts its counterpart on constructible sheaves. Our shift convention is that $\psi_f[-1]$ is exact on mixed Hodge modules/perverse sheaves. 

\subsubsection{}Write $\psi_f^u$ for the unipotent part of $\psi_f$. The log of the unipotent part of the monodromy operator on nearby cycles yields a canonical map:
\[ N \colon \psi_f^u(M) \to \psi_f^u(M)(-1) \]
in $\DM(X_0)$.
Here, $M(-1) = M\boxtimes \QQ(-1)$, where $\QQ(-1)$ is the one-dimensional Hodge structure of type $(1,1)$.
This map fits into a canonical distinguished triangle (see \cite[Remark 5.2.2]{S88}):
\[ i^*j_*j^*(M) \to \psi_f^u(M) \mapright{N} \psi_f^u(M)(-1) \mapright{[1]} \]
which in turn yields a long exact sequence of mixed Hodge structures:
\begin{equation}\label{monles}
\cdots \to H^k(X_0; i^*j_*j^*M) \to H^k(X_0; \psi_f^u(M)) \mapright{N} H^k(X_0;\psi_f^u(M))(-1) \to \cdots
\end{equation}

\subsection{Weight considerations}
\subsubsection{}
Let $k\in\ZZ$.
Given a nilpotent endomorphism $N$, of an object $V$ in an abelian category,
there exists a unique finite increasing filtration $V_{\bullet}$ of $V$, such that $NV_i \subset NV_{i-2}$ and such that $N^i$ induces an isomorphism: $\Gr_{k+i}V\mapright{\sim} \Gr_{k-i}V$. This is the \emph{monodromy filtration} of \cite[1.6]{D} centered at $k$.

\subsubsection{}\label{monwts}
Let $f\colon X\to \AA^1$ be a proper morphism. If $M\in\DM(X)$ is pure of weight $0$, then the weight filtration on $H^k(X_0; \psi_f^u(M))$ coincides with the monodromy filtration determined by $N$ and centered at $k$. In particular:
\begin{enumerate}
\item $\ker(N\colon H^k(X_0; \psi_f^u(M)) \to H^k(X_0; \psi_f^u(M))(-1))$ has weights $\leq k$;
\item $\coker(N\colon H^k(X_0; \psi_f^u(M)) \to H^k(X_0; \psi_f^u(M))(-1))$ has weights $\geq k+2$.
\end{enumerate}

\subsubsection{}
By \eqref{relles} and \eqref{monles} we have a diagram:
\begin{equation}\label{dia1}\begin{gathered}\xymatrixrowsep{7mm}\xymatrix{
H^k(X_0; i^*M)\ar[d] \\
H^k(X_0; i^*j_*j^*M) \ar[d]\ar[r] & H^k(X_0;\psi^u_f(M))\ar[r]^-N & H^k(X_0;\psi_f^u(M))(-1) \\
H^{k+1}(X_0; i^!M) }
\end{gathered}\end{equation}
where both the row and column are exact.
\begin{prop}[Local invariant cycles]\label{1prop}
Let $M\in \DM(X)$ be pure. If $f\colon X\to \AA^1$ is proper, then the sequence:
\[ H^k(X_0; i^*M) \to H^k(X_0; \psi_f^u(M)) \mapright{N} H^k(X_0; \psi_f^u(M))(-1) \]
is exact for each $k$.
\end{prop}

\begin{proof}
We may assume $M$ is pure of weight $0$. Then $\ker(N)$ has weights $\leq k$ by \S\ref{monwts}(i).
Consequently, by the exactness of the row in \eqref{dia1}, it suffices to show $H^k(X_0; i^*M) \to H^k(X_0; i^*j_*j^*M)$ is surjective on weights $\leq k$.
As $i^!$ does not lower weights, $H^{k+1}(X_0;i^!M)$ has weights $\geq k+1$. In view of the exactness of the column in \eqref{dia1}, this yields the desired surjectivity.
\end{proof}

\subsubsection{}From \eqref{relles} and \eqref{monles} we also have a diagram:
\begin{equation}\label{dia2}\begin{gathered}\xymatrixrowsep{7mm}\xymatrix{
&& H^{k+1}(X_0;i^*M) \ar[d] \\
H^k(X_0; \psi_f^u(M)) \ar[r]^-N & H^k(X_0; \psi_f^u(M))(-1) \ar[r] & H^{k+1}(X_0; i^*j_*j^*M)\ar[d] \\
&& H^{k+2}(X_0;i^!M)
}\end{gathered}\end{equation}
where both the row and column are exact.
\begin{prop}\label{2prop}Let $M\in \DM(X)$ be pure. If $f\colon X \to \AA^1$ is proper, then:
\[ H^k(X_0; \psi^u_f(M)) \mapright{N} H^k(X_0; \psi^u_f(M))(-1) \to H^{k+2}(X_0; i^!M) \]
is exact for each $k$.
\end{prop}

\begin{proof}We may assume $M$ is pure of weight $0$. Then $\coker(N)$ has weights $\geq k+2$ by \S\ref{monwts}(ii). 
Consequently, by the exactness of the row in \eqref{dia2}, it suffices to show $H^{k+1}(X_0; i^*j_*j^*M) \to H^{k+2}(X_0; i^!M)$ is injective on weights $\geq k+2$.
As $i^*$ does not raise weights, $H^{k+1}(X_0;i^*M)$ has weights $\leq k+1$. In view of the exactness of the column in \eqref{dia2}, this yields the desired injectivity. 
\end{proof}

\subsubsection{}
Next, from \eqref{relles} and \eqref{monles} once again, we have the diagram:
\begin{equation}\label{dia3}\begin{gathered}\xymatrixrowsep{7mm}\xymatrixcolsep{7mm}\xymatrix{
H^k(X_0;\psi_f^u(M))(-1) \ar[r] & H^{k+1}(X_0;i^*j_*j^*M)\ar[r]\ar[d] & H^{k+1}(X_0;\psi_f^u(M)\ar[r]^-N& \\
& H^{k+2}(X_0; i^!M)\ar[d] \\
& H^{k+2}(X_0; i^*M)
}\end{gathered}\end{equation}
where both the row and column are exact.
\begin{prop}\label{3prop}
Let $M\in \DM(X)$ be pure. If $f$ is proper, then:
\[ H^k(X_0; \psi^u_f(M))(-1) \to H^{k+2}(X_0; i^!M) \to H^{k+2}(X_0;i^*M) \]
is exact for each $k$.
\end{prop}

\begin{proof}
We may assume $M$ is pure of weight $0$. Then $H^{k+2}(X_0; i^!M)$ has weights $\geq k+2$. Consequently, by the exactness of the column in \eqref{dia3}, it suffices to show 
$H^{k}(X_0; \psi_f^u(M))(-1) \to H^{k+1}(X_0; i^*j_*j^*M)$
is surjective on weights $\geq k+2$. 
By \S\ref{monwts}(i), $\ker(N)$ has weights $\leq k+1$. Hence, by the exactness of the row in \eqref{dia3}, the map
$H^{k+1}(X_0; i^*j_*j^*M) \to H^{k+1}(X_0; \psi^u_f(M))$
 must be $0$ on weights $\geq k+2$. So row exactness of \eqref{dia3} yields the desired surjectivity.
\end{proof}

\subsubsection{}
Finally, using \eqref{relles} and \eqref{monles} (one last time), we have the diagram:
\begin{equation}\label{dia4}\begin{gathered}\xymatrixrowsep{7mm}\xymatrix{
&&\ar[d]^N \\
&& H^{k-1}(X_0; \psi_f^u(M))(-1)\ar[d] \\
H^k(X_0;i^!M) \ar[r]& H^k(X_0; i^*M)\ar[r] & H^k(X_0; i^*j_*j^*M) \ar[d] \\
&& H^k(X_0; \psi_f^u(M))
}\end{gathered}\end{equation}
where both the row and column are exact.
\begin{prop}\label{4prop}
Let $M\in \DM(X)$ be pure. If $f$ is proper, then:
\[ H^{k}(X_0; i^!M) \to H^{k}(X_0; i^*M) \to H^{k}(X_0;\psi_f^u(M)) \]
is exact for each $k$.
\end{prop}

\begin{proof}
We may assume $M$ is pure of weight $0$. Then $H^k(X_0;i^*M)$ has weights $\leq k$. Consequently, by the exactness of the row in \eqref{dia4}, it suffices to show that 
$H^k(X_0; i^*j_*j^*M) \to H^k(X_0; \psi_f^u(M))$
is injective on weights $\leq k$. By \S\ref{monwts}(ii),
\[ \coker(N\colon H^{k-1}(X_0;\psi_f^u(M)) \to H^{k-1}(X_0;\psi_f^u(M))(-1))\]
has weights $\geq k+1$. So column exactness of \eqref{dia4} yields the desired injectivity.
\end{proof}

\begin{thm}[Generalized Clemens-Schmid]\label{csthm}
Let $M\in \DM(X)$ be pure. If $f\colon X\to\AA^1$ is proper, then we have (two) long exact sequences:
\[ \cdots \to H^k(X_0, i^*M) \to H^k(X_0;\psi_f^u(M)) \mapright{N} H^k(X_0; \psi_f^u(M))(-1) \to H^{k+2}(X_0; i^!M) \to \cdots \]
\end{thm}

\begin{proof}This is just Propositions \ref{1prop}, \ref{2prop}, \ref{3prop} and \ref{4prop} stated in combined form.
\end{proof}

\subsection{Applications}
\subsubsection{}Let $a\colon X\to \Spec(\CC)$ be the evident map. Let $\QQ$ denote the trivial one-dimensional Hodge structure of weight $0$. Then setting:
\begin{align*}
H^*(X) &= H^*(X; a^*\QQ), \\
H^*_{X_0}(X) &= H^*(X; i^!a^*\QQ),
\end{align*}
yields canonical mixed Hodge structures on the usual cohomology groups (with $\QQ$-coefficients) on the left. Similarly, set:
\[ H^*(X_{\infty}) = H^*(X_0; \psi_f(a^*\QQ)). \]
Then $H^*(X_{\infty})$ is the limit Hodge structure associated to $f\colon X\to \AA^1$. If $f$ is proper, then forgetting Hodge structures, $H^*(X_{\infty})$ is isomorphic to $H^*(f^{-1}(t))$, for $t$ sufficiently close to $0$. However, this is usually not an isomorphism of Hodge structures.
Similarly, if we let $T$ be the monodromy operator on $H^*(X_{\infty})$, then $T$ usually does \emph{not} preserve Hodge structures. However, $T-\id$ and $N$ (the log of the unipotent part of $T$) have the same kernel (monodromy invariants).

\begin{thm}[Local invariant cycles]\label{invthm}Assume $X$ is rationally smooth. If $f\colon X \to \AA^1$ is proper, then for each $k$ we have an exact sequence of mixed Hodge structures:
\[ H^{k-2}_{X_0}(X)\to H^k(X_0) \to H^k(X_{\infty})^T \to 0, \]
where $(-)^T$ denotes monodromy invariants.
\end{thm}

\begin{proof}
As $X$ is rationally smooth, $a^*\QQ$ is pure (ignoring cohomological degree, it is an irreducible mixed Hodge module with underlying perverse sheaf the intersection cohomology complex of $X$). Thus, Theorem \ref{csthm} (or Proposition \ref{1prop}) applies.
\end{proof}

\subsubsection{}
Using this recipe we may re-state Theorem \ref{csthm} under various guises. Let's just close with the most straightforward one:
\begin{thm}[Clemens-Schmid]\label{uncsthm}Assume $X$ is rationally smooth and that the monodromy operator acts unipotently. If $f\colon X\to \AA^1$ is proper, then:
\[ \cdots \to H^k(X_0) \to H^k(X_{\infty}) \mapright{N}H^k(X_{\infty})(-1) \to H^{k+2}_{X_0}(X) \to \cdots \]
is an exact sequence of mixed Hodge structures.
\end{thm}

\begin{proof}
As $X$ is rationally smooth, $a^*\QQ$ is pure. Thus, Theorem \ref{csthm} applies.
\end{proof}

\end{document}